\documentclass[10pt, reqno]{amsart}
\usepackage{amsmath, amsthm, amscd, amsfonts, amssymb, graphicx, color}
\usepackage[bookmarksnumbered, colorlinks, plainpages]{hyperref}
\hypersetup{colorlinks=true,linkcolor=red, anchorcolor=green, citecolor=cyan, urlcolor=red, filecolor=magenta, pdftoolbar=true}
\usepackage{mathrsfs}

\newtheorem{theorem}{Theorem}[section]
\newtheorem{lemma}[theorem]{Lemma}

\newtheorem{corollary}[theorem]{Corollary}
\theoremstyle{definition}

\newtheorem{example}[theorem]{Example}

\newtheorem{conjecture}[theorem]{Conjecture}

\theoremstyle{remark}
\newtheorem{remark}[theorem]{Remark}
\numberwithin{equation}{section}

\begin{document}
\setcounter{page}{1}

\title[Function field analogy]{Remark on function field analogy}

\author[Nikolaev]
{Igor V. Nikolaev$^1$}

\address{$^{1}$ Department of Mathematics and Computer Science, St.~John's University, 8000 Utopia Parkway,  
New York,  NY 11439, United States.}
\email{\textcolor[rgb]{0.00,0.00,0.84}{igor.v.nikolaev@gmail.com}}





\begin{abstract}
We study the analogy between  number fields and function fields in one variable over  finite
fields.  The main result is an isomorphism between the Hilbert class fields of class number one
and a family of the function fields $\mathbf{F}_q(C)$ over a desingularized algebraic curve $C$. 
Our proof is based on  the K-theory of  the  Serre $C^*$-algebras and  birational geometry
of the curve $C$.  We apply the isomorphism to construct explicit generators of the Hilbert class fields coming 
from the torsion submodules of the Drinfeld module.  
\end{abstract}

\maketitle

\section{Introduction}
It was first observed in [Dedekind \& Weber 1882] \cite{DedWeb1}
that the number fields and the fields of rational functions in one variable
are related. For example,   both fields  are  the Dedekind domains, i.e.  every ideal factors 
into a product of the prime ideals.  Unlike the number fields,  the group of units of the function fields
 is no longer finitely generated.  However, the behavior is back to normal for 
 the rational functions over finite (Galois)  fields. 
In fact,  it was proved in [Artin \& Whaples 1945]  \cite{ArtWap1} that 
the two fields can be unified as  global fields. 
This led to a fast progress  driven by solution of the  open 
problems on either side by recasting them in terms of the other [van der Geer,  Moonen \& Schoof 2005] \cite{GMS}.
The successful proof  of an  analog of the Riemann Hypothesis (Hasse \& Weil) and the geometric Langlands Conjectures (Drinfeld \& Lafforgue)
are few   examples.  When and if  such an analogy can be upgraded to an isomorphism is an interesting open problem.

The aim of our note is a solution of the above problem for the Hilbert class fields of class number one,
i.e.  the principal ideal domains.  Namely, we construct an isomorphism between such fields  and the function
fields  $\mathbf{F}_{p^{m-1}}(C)$,  where $C$ is a smooth algebraic curve obtained from the singular curve 
by the $m-1$ blow-ups (Theorem \ref{thm1.1}).  To formalize our results,   the following notation will be used.

Let  $(x_0,y_0)$ be an isolated singularity of  algebraic curve 
$ f(x,y)=0$ given by a polynomial  $f\in \mathbf{Z}[x,y]$.  Consider an arithmetic scheme
$Spec ~\mathbf{Z}[x,y]/(f)$ and its natural morphism to $Spec~\mathbf{Z}$.  
Let $\mathscr{P}\in Spec ~\mathbf{Z}[x,y]/(f)$ be the maximal  ideal corresponding to  the point $(x_0,y_0)$
lying over $(p)\in Spec~\mathbf{Z}$ for a prime $p\in\mathbf{Z}$.

Recall that the singularity $(x_0,y_0)$ can be resolved by a sequence 
of the blow-ups  $\{\epsilon_i ~|~1\le i\le m-1\}$,  such  that $f(x,y)=0$ 
defines a smooth algebraic curve $C$, see e.g.  [Grassi 2009] \cite[Example-Definition 10]{Gra1} or
 [Koll\'ar 2007] \cite[Section 1.7]{K}. 
We denote by $\mathbf{F}_{p^{m-1}}(C)$  the geometric extension of the field  of rational functions in one variable 
over the Galois field  $\mathbf{F}_{p^{m-1}}$ [Rosen 2002] \cite[Chapter 7]{R}.

On the other hand, the K-theory of the Serre $C^*$-algebra of an algebraic  curve gives rise to 
a number field $k$ \cite[Section 5.3.1]{N}.  
Moreover, if the curve was singular,  then its blow-up corresponds to the Hilbert class field $\mathscr{H}(k)$,
i.e. the maximal abelian unramified extension of  $k$.  This is an implication of a general result 
for the birational maps between the algebraic surfaces  \cite[Theorem 1.3]{Nik1}.
Let  $\mathscr{H}^{i+1}(k):=\mathscr{H}(\mathscr{H}^i(k))$. 
If $C$ is a curve desingularized  after $m-1$ blow-ups, then the 
number field  $\mathscr{H}^{m-1}(k)$ must have class number one,  since each birational map  
between smooth curves is an isomorphism. 
Our main result can be formulated as follows. 

\medskip
\begin{theorem}\label{thm1.1}
$\mathscr{H}^{m-1}(k)\cong \mathbf{F}_{q}(C),$ where  $q=p^{m-1}$.
\end{theorem}

\medskip
The paper is organized as follows.  A brief review of the preliminary facts is 
given in Section 2. Theorem \ref{thm1.1} is proved in Section 3. 
We consider an application of Theorem \ref{thm1.1} in Section 4,
where  explicit generators of the Hilbert class fields coming 
from the torsion submodules of the Drinfeld module are constructed.

\section{Preliminaries}
This section contains a brief review of the Serre $C^*$-algebras,  birational geometry
and function fields.
We refer the reader to  [Grassi 2009] \cite{Gra1},  [Koll\'ar 2007] \cite{K},  [Rosen 2002] \cite{R} and
 [Stafford \& van den Bergh 2001] \cite{StaVdb1}  for a detailed account.

\subsection{Serre $C^*$-algebras (\cite{N}, \cite{StaVdb1})}
Let $V$ be a projective variety over the field $k$.  Denote by $\mathcal{L}$ an invertible
sheaf of the linear forms on $V$.  If $\sigma$ is an automorphism of $V$,  then
the pullback of $\mathcal{L}$ along $\sigma$ will be denoted by $\mathcal{L}^{\sigma}$,
i.e. $\mathcal{L}^{\sigma}(U):= \mathcal{L}(\sigma U)$ for every $U\subset V$. 
The graded $k$-algebra
\begin{equation}\label{eq2.1}
B(V, \mathcal{L}, \sigma)=\bigoplus_{i\ge 0} H^0\left(V, ~\mathcal{L}\otimes \mathcal{L}^{\sigma}\otimes\dots
\otimes  \mathcal{L}^{\sigma^{ i-1}}\right)
\end{equation}
is called a  twisted homogeneous coordinate ring of $V$ [Stafford \& van den Bergh 2001]  \cite{StaVdb1}.  Such a ring is 
always non-commutative,  unless the automorphism $\sigma$ is trivial. 
A multiplication of sections of  $B(V, \mathcal{L}, \sigma)=\oplus_{i=1}^{\infty} B_i$ is defined by the 
rule  $ab=a\otimes b$,   where $a\in B_m$ and $b\in B_n$.
An invertible sheaf $\mathcal{L}$ on $V$  is called $\sigma$-ample, if for 
every coherent sheaf $\mathcal{F}$ on $V$,
 the cohomology group $H^k(V, ~\mathcal{L}\otimes \mathcal{L}^{\sigma}\otimes\dots
\otimes  \mathcal{L}^{\sigma^{ n-1}}\otimes \mathcal{F})$  vanishes for $k>0$ and
$n>>0$.   If $\mathcal{L}$ is a $\sigma$-ample invertible sheaf on $V$,  then
\begin{equation}\label{eq2.2}
Mod~(B(V, \mathcal{L}, \sigma)) / ~Tors ~\cong ~Coh~(V),
\end{equation}
where  $Mod$ is the category of graded left modules over the ring $B(V, \mathcal{L}, \sigma)$,
$Tors$ is the full subcategory of $Mod$ of the torsion  modules and  $Coh$ is the category of 
quasi-coherent sheaves on a scheme $V$.  In other words, the $B(V, \mathcal{L}, \sigma)$  is  
a coordinate ring of the variety $V$.

Let $R$ be a commutative  graded ring,  such that $V=Proj~(R)$.  
Denote by $R[t,t^{-1}; \sigma]$
the ring of skew Laurent polynomials defined by the commutation relation
$b^{\sigma}t=tb$  for all $b\in R$, where $b^{\sigma}$ is the image of  $b$ under automorphism 
$\sigma$.  It is known, that $R[t,t^{-1}; \sigma]\cong B(V, \mathcal{L}, \sigma)$ [Stafford \& van den Bergh 2001]  \cite[Section 5]{StaVdb1}.
Let $\mathcal{H}$ be a Hilbert space and   $\mathscr{B}(\mathcal{H})$ the algebra of 
all  bounded linear  operators on  $\mathcal{H}$.
For a  ring of skew Laurent polynomials $R[t, t^{-1};  \sigma]$,  
 consider a homomorphism: 
\begin{equation}\label{eq2.3}
\rho: R[t, t^{-1};  \sigma]\longrightarrow \mathscr{B}(\mathcal{H}). 
\end{equation}
Recall  that  $\mathscr{B}(\mathcal{H})$ is endowed  with a $\ast$-involution;
the involution comes from the scalar product on the Hilbert space $\mathcal{H}$. 
We shall call representation (\ref{eq2.3})  $\ast$-coherent,   if
(i)  $\rho(t)$ and $\rho(t^{-1})$ are unitary operators,  such that
$\rho^*(t)=\rho(t^{-1})$ and 
(ii) for all $b\in R$ it holds $(\rho^*(b))^{\sigma(\rho)}=\rho^*(b^{\sigma})$, 
where $\sigma(\rho)$ is an automorphism of  $\rho(R)$  induced by $\sigma$. 
Whenever  $B=R[t, t^{-1};  \sigma]$  admits a $\ast$-coherent representation,
$\rho(B)$ is a $\ast$-algebra.  The norm closure of  $\rho(B)$  is   a   $C^*$-algebra
   denoted  by $\mathscr{A}_V$.  We  refer to  $\mathscr{A}_V$  as   the   {\it Serre $C^*$-algebra}
 of  $V$ \cite[Section 5.3.1]{N}. 

Let $(K_0(\mathscr{A}_V), K_0^+(\mathscr{A}_V), u)$ be the dimension group  of the 
$C^*$-algebra $\mathscr{A}_V$ [Blackadar 1986] \cite[Sections 6.3 and 7.3]{B}. 
Such a group defines the Serre $C^*$-algebra $\mathscr{A}_V$ up to an isomorphism. 
Moreover, if $V$ is defined over a number field,  then its dimension group is order-isomorphic
to the dimension group $(O_K, O_K^+, [u])$, where  $O_K$ is the ring of integers of a number field $K$,
$O_K^+$ the additive semigroup of positive integers for a real embedding 
$O_K\hookrightarrow\mathbf{R}$ and $[u]$ is a unit of  $O_K$ \cite[pp. 194-195]{N}.

\subsection{Birational geometry (\cite{Gra1}, \cite{Nik1})}
Let $V$ be a projective variety over the field $k$. 
The map $\phi: V\dashrightarrow V'$ is  rational,  if it is represented
by rational functions. 
The rational maps cannot be composed unless they are dominant, i.e. the image 
of $\phi$ is Zariski dense in $V'$. 
The map $\phi$ is called birational,  if the inverse $\phi^{-1}$ is a rational map.

A birational map  $\epsilon: V\dashrightarrow V'$ is called a blow-up,
if it is defined everywhere except for a point $p\in V$ and an exceptional divisor $E\subset V'$,
such  that  $\epsilon^{-1}(V)=p$.   
Every birational map $\phi: V\dashrightarrow V'$ is composition of a finite 
number of the blow-ups, i.e. $\phi=\epsilon_1\circ\dots\circ\epsilon_k$.

The variety $V$ is called a minimal model, if any birational map $V\dashrightarrow V'$
is an isomorphism.  In case $V$ is a  curve or a surface, it is known that every
birational class has a smooth minimal model. 
Namely, if $S$ is a surface, the minimal models exist and are unique unless $S$ is 
a ruled surface. By the Castelnuovo Theorem, the smooth projective surface $S$ is a minimal 
model if and only if $S$ does not contain a rational curves $C$ with the 
self-intersection index $-1$.  
If $C$ is a curve, then each birational map is an isomorphism, i.e. $C$ is the minimal model.

 It is still not known if any variety has
a minimal model; it has been shown that in dimension 3 and higher, the minimal
model is not necessarily smooth.  The minimal model
always exists for certain types of varieties.

The birational map  $V\dashrightarrow V'$ induces a morphism
 $\mathscr{A}_V\dashrightarrow \mathscr{A}_{V'}$ between the
 corresponding Serre $C^*$-algebras, see Section 2.1. 
The following result describes such a morphism 
for the varieties $V$  over number fields.
\begin{theorem}\label{thm2.1}
 {\bf (\cite[Theorem 1.3]{Nik1})}
 The map $\mathscr{A}_V\dashrightarrow \mathscr{A}_{V'}$
 gives rise to an extension of 
 the corresponding dimension groups 
 $(O_k, O_k^+, [u])\dashrightarrow (O_{\mathscr{H}(k)}, O_{\mathscr{H}(k)}^+, [u])$,
 where $\mathscr{H}(k)$ is the Hilbert class field of $k$, i.e. the maximal abelian unramified 
 extension of the field $k$. 
\end{theorem}
\begin{remark}
Theorem \ref{thm2.1} was proved for $V$ being a surface $S$ 
or a curve $C$ \cite{Nik1}. The proof  extends  to the higher dimensions,  but the topic 
lies  outside the scope of our paper. In what follows, we use Theorem \ref{thm2.1}
only for the case $V$ being a curve $C$. 
\end{remark}

\subsection{Function fields (\cite{R})}
We briefly review the fields  of rational functions in one variable 
over the finite  fields $\mathbf{F}_q$.  
\subsubsection{The field $\mathbf{F}_q(T)$}
Let $p$ be a prime number and $q=p^n$ for a positive integer $n$. 
Denote by $\mathbf{F}_q$ the field with $q$ elements, i.e. the field 
$\mathbf{Z}/p^n\mathbf{Z}$.  Let  $\mathbf{F}_q[T]$ be the polynomial 
ring over $\mathbf{F}_q$ and let  $\mathbf{F}_q(T)$ be the  field of fractions of   $\mathbf{F}_q[T]$. 
The ring  $\mathbf{F}_q[T]$ has many properties 
in common with the ring of integers $\mathbf{Z}$. Both are principal ideal 
domains, both have a finite group of units, and every residue class ring modulo
a non-zero ideal is finite.  Namely, if $\deg g$ is the degree
of the polynomial $g\in  \mathbf{F}_q[T]$, then
\begin{equation}\label{eq2.4}
 |\mathbf{F}_q[T]/ g \mathbf{F}_q[T]|=q^{\deg g}.
\end{equation}

\subsubsection{Geometric extensions}
Consider an extension of the field $\mathbf{F}_q(T)$ of the transcendence degree one. 
Any such has the form $\mathbf{F}_q(T)(y)$, where $y$ satisfies a polynomial equation:
\begin{equation}\label{eq2.5}
f(x,y)=0,
 \end{equation}
where $x=T$. 
In other words, the  $\mathbf{F}_q(T)(y)$ is a field of the rational function on
the algebraic curve (\ref{eq2.5}).  The extension  $\mathbf{F}_q(T)(y)$ is called 
geometric, since it is independent of the extensions of the field of constants $\mathbf{F}_q$. 
The smooth algebraic curve (\ref{eq2.5}) is denoted by $C$ and the corresponding geometric 
extension can be written as  $\mathbf{F}_q(C)$.   There exists a bijection between 
all geometric extension of the field  $\mathbf{F}_q(T)$ and   algebraic curves $C$ 
given by equation (\ref{eq2.5})  for a polynomial  $f\in\mathbf{Z}[x,y]$.

\subsubsection{Drinfeld module}
The explicit class field theory for the function fields is strikingly simpler
than for the number fields. The generators of the maximal abelian unramified
extensions (i.e. the Hilbert class fields) are constructed using 
the concept of the Drinfeld module. Roughly speaking, such a module 
is an analog of the exponential function and a generalization of the Carlitz module. 
Nothing similar  exists at the number fields side, where  the explicit 
generators of abelian extensions are known only for the field of rationals
(Kronecker-Weber theorem)
and imaginary quadratic number fields (complex multiplication). 
Below we give some details
on the Drinfeld modules.

Let $k$ be a field.  A polynomial $f\in k[x]$ is said to be additive
in the ring $k[x,y]$ if $f(x+y)=f(x)+f(y)$. If $char ~k=p$, then it is verified 
directly that the polynomial $\tau(x)=x^p$ is additive. Moreover, each 
additive polynomial has the form $a_0x+a_1x^p+\dots+a_rx^{p^r}$. 
The set of all additive polynomials is closed under addition and composition 
operations. The corresponding ring is isomorphic to a ring $k\langle\tau\rangle$
of the non-commutative polynomials given by  the commutation relation:
\begin{equation}
\tau a=a^p\tau, \qquad \forall a\in k. 
\end{equation}

Let $A=\mathbf{F}_q[T]$ and $k=\mathbf{F}_q(T)$. 
By the  Drinfeld module one understands a homomorphism
\begin{equation}\label{eq2.7}
\rho: A\to k\langle\tau\rangle,
\end{equation}
such that for all $a\in A$ the constant term of $\rho_a$ is $a$ and 
$\rho_a\not\in k$ for at least one $a\in A$.  
Consider a torsion module $\Lambda_{\rho}[a]=\{\lambda\in\bar k ~|~\rho_a(\lambda)=0\}$.
The following result describes the simplest case of the explicit class field theory for the function
fields.
\begin{theorem}\label{thm2.3}
For each non-zero $a\in A$ the function field $k\left(\Lambda_{\rho}[a]\right)$ 
is an abelian extension of $k$ whose Galois group is isomorphic to a subgroup
of $\left(A/(a)\right)^*$. 
\end{theorem}

\section{Proof}
For the sake of clarity, let us outline the main ideas. 
Consider the case $m=2$ of Theorem \ref{thm1.1},  while the rest is proved by induction.
One starts with  a non-singular algebraic curve $C$ 
given by the equation $f(x,y)=0$, where  $f\in\mathbf{Z}[x,y]$. The curve is obtained by the 
blow-up $C_{sing}\buildrel p\over\dashrightarrow C$ of a singular curve $C_{sing}$ at the point corresponding
to the ideal $(p)\in Spec ~\mathbf{Z}$ under  the map  $Spec ~\mathbf{Z}[x,y]/(f)\to  Spec ~\mathbf{Z}$. 
Since $C$ is a smooth curve given by a polynomial $f\in\mathbf{Z}[x,y]$, it defines a geometric 
extension $\mathbf{F}_p(C)$ of the function field $\mathbf{F}_p(T)$. 
In fact, such an extension is functorial, i.e. morphisms of $C$ correspond to the isomorphisms
of the field   $\mathbf{F}_p(C)$.  On the other hand, the K-theory of the Serre $C^*$-algebra of the 
curve $C$  gives rise to a number field $\mathscr{H}(k)$ of the class number one,
see Sections 2.1 and 2.2. Such an assignment is  also functorial. Thus one gets a pair 
of the global fields  $\mathbf{F}_p(C)$ and  $\mathscr{H}(k)$  attached to curve $C$ (and its Jacobian variety $Jac~C$)
as it is shown in the commutative diagram in Figure 1. 
Since $Jac~C$ is an abelian variety, the horizontal arrow in Figure 1 is realized by
 an isomorphism between the two fields. Let us pass to a detailed
argument by splitting the proof in a series of lemmas.

\begin{figure}[h]
\begin{picture}(150,150)(0,30)

\put(60,150){$Jac ~C$}

\put(73,140){\vector(0,-1){25}}
\put(60,140){\vector(-1,-2){35}}
\put(88,140){\vector(1,-2){35}}

\put(70,100){$C$}

\put(65,95){\vector(-1,-1){25}}
\put(85,95){\vector(1,-1){25}}
\put(50,53){\vector(1,0){50}}

\put(70,55){$\sim$}

\put(10,50){$\mathscr{H}(k)$}
\put(110,50){$\mathbf{F}_p(C)$}

\end{picture}
\caption{
}
\end{figure}

\begin{lemma}\label{lm3.1}
Each isogeny of the Jacobian variety $Jac ~C$ of curve $C$ gives 
rise to a multiplication $\alpha O_{\mathscr{H}(k)}$ 
 ($\beta O_{\mathbf{F}_p(C)}$, resp.) of the ring of integers 
 of the field $\mathscr{H}(k)$ ($\mathbf{F}_p(C)$, resp.) 
 by an element $\alpha\in O_{\mathscr{H}(k)}$
  ($\beta\in O_{\mathbf{F}_p(C)}$, resp.)
  Moreover, the degree of such isogeny is equal to the norm 
  of the element $\alpha$ ($\beta$, resp.) 
\end{lemma}
\begin{proof}
(i) Let $C$ be an algebraic curve defining the function field $\mathbf{F}_p(C)$. 
Instead of $C$,  we shall work with its Jacobian variety $Jac~C$, i.e. a complex torus
\begin{equation}\label{eq3.1}
 Jac~C=\mathbf{C}^g/\Lambda,
\end{equation}
 where $g$ is the genus of $C$ and $\Lambda$ is a lattice of rank  $2g$. 
Since $Jac~C$ defines (and is defined by) the curve $C$, one gets a commutative diagram 
in Figure 1.

\bigskip
(ii) Denote by $f$  an isogeny of the abelian variety $Jac~C$. Recall that the map $f$ defines a sub-lattice 
$\Lambda_f\subseteq\Lambda$ of a finite index,  such that 
\begin{equation}\label{eq3.2}
 f(Jac~C)=\mathbf{C}^g/\Lambda_f,
\end{equation}
where the degree of  $f$ is equal to the index of $\Lambda_f$ in the lattice $\Lambda$. 
We shall consider the action of isogeny $f$ on the fields  $\mathscr{H}(k)$  and $\mathbf{F}_p(C)$,
respectively.

\bigskip
(iii)  Let  $O_{\mathscr{H}(k)}$ be the ring of integers of the number field  $\mathscr{H}(k)$. 
Denote by  $f$  an isogeny of the Jacobian variety $Jac~C$. 
Let $(O_{\mathscr{H}(k)}, O_{\mathscr{H}(k)}^+, [u])$  be the dimension group of the corresponding
Serre $C^*$-algebra,  see Section 2.1 for the notation.  Recall that the lattice $\Lambda$ in
(\ref{eq3.1}) has complex multiplication, i.e. the lattice in (\ref{eq3.2}) can be written in the form:
\begin{equation}\label{eq3.5}
\Lambda_f=\omega\Lambda,
\end{equation}
where $\omega$ is a complex number in a CM-field. 
On the other hand, it is known that the functor  $Jac~C\to \mathscr{H}(k)$
transforms complex multiplication to a real multiplication, i.e. 
the dimension group of the variety (\ref{eq3.2}) has the form:
\begin{equation}\label{eq3.6}
(\alpha O_{\mathscr{H}(k)}, \alpha O_{\mathscr{H}(k)}^+, [u])
\end{equation}
for an element $\alpha\in O_{\mathscr{H}(k)}$ \cite[Section 1.4.1]{N}.
It follows from formulas (\ref{eq3.5}) and (\ref{eq3.6}) that
each isogeny $f$ defines a correspondence   $f\mapsto \alpha\in O_{\mathscr{H}(k)}$.

\bigskip
 (iv) Let $O_{\mathbf{F}_p(C)}$ be the ring of integers of the function field  $\mathbf{F}_p(C)$. 
Denote by  $f$  an isogeny of the Jacobian variety $Jac~C$. Recall that 
the map $Jac~C\to \mathbf{F}_p(C)$ is functorial and therefore must preserve  
the inclusion relation of the corresponding lattices: 
\begin{equation}\label{eq3.3}
\Lambda_f\subseteq\Lambda,
\end{equation}
see (\ref{eq3.1}) and (\ref{eq3.2}).   In particular, the inclusion (\ref{eq3.3})
gives rise to a subring $R_f\subseteq O_{\mathbf{F}_p(C)}$. It is easy to see,  that  the $R_f$ 
is an ideal of the ring $O_{\mathbf{F}_p(C)}$.  Indeed, $\Lambda /\Lambda_f$ is an abelian group
and therefore $\Lambda_f$ is the kernel of a homomorphism of $\Lambda\to\Lambda /\Lambda_f$.
Thus  the subring $R_f\subseteq O_{\mathbf{F}_p(C)}$ is the kernel of a ring homomorphism. 
In particular, the $R_f$ is an ideal of   the ring  $O_{\mathbf{F}_p(C)}$. 
But the ring $O_{\mathbf{F}_p(C)}$ was assumed to have the class number one, i.e. each ideal of 
 $O_{\mathbf{F}_p(C)}$ is principal. Thus there exists an element $\beta\in O_{\mathbf{F}_p(C)}$,
 such that
\begin{equation}\label{eq3.4}
R_f=\beta O_{\mathbf{F}_p(C)}. 
\end{equation}
In other words, the isogeny $f$  in formula  (\ref{eq3.4}) 
gives rise to a correspondence $f\mapsto \beta\in O_{\mathbf{F}_p(C)}$.

\bigskip
 (v)  Finally, let us show that the degree of isometry $f$ is equal to the norm of  $\alpha\in O_{\mathscr{H}(k)}$
 ($\beta\in O_{\mathbf{F}_p(C)}$, resp.)  Indeed, $deg~(f)=|\Lambda /\Lambda_f|=N(\alpha)$, where $N(\alpha)$
 is the norm of the algebraic integer  $\alpha\in O_{\mathscr{H}(k)}$.  The argument for the function fields is similar, 
 where  the norm $N(\beta)$ is calculated by the formula (\ref{eq2.4}), i.e. $N(\beta)=p^{deg~\beta}$. 

\bigskip
Lemma \ref{lm3.1} is proved. 
\end{proof}

\begin{lemma}\label{lm3.2}
 $O_{\mathscr{H}(k)}\cong  O_{\mathbf{F}_p(C)}.$
\end{lemma}
\begin{proof}
(i) Lemma \ref{lm3.1} says that each isogeny $f$ of $Jac~C$ defines the elements
$\alpha\in O_{\mathscr{H}(k)}$ and $\beta\in O_{\mathbf{F}_p(C)}$. Consider 
a map:
\begin{equation}\label{eq3.7}
O_{\mathscr{H}(k)}\to O_{\mathbf{F}_p(C)}
\end{equation}
given by the formula $\alpha\mapsto\beta$ as $f$ runs all possible isogenies of the 
variety $Jac~C$. 

\bigskip
(ii) On the other hand, the abelian variety $Jac ~C$ has complex multiplication
and therefore each isogeny is represented by an element of the maximal commutative subring
of the endomorphism ring of $Jac ~C$. (The latter is isomorphic to an order in the  ring of integers of the corresponding CM-field.)
Thus the map  (\ref{eq3.7}) is a ring isomorphism.

\bigskip
Lemma \ref{lm3.2} is proved. 
\end{proof}

\begin{corollary}\label{cr3.3}
 $\mathscr{H}(k)\cong\mathbf{F}_p(C).$
\end{corollary}
\begin{proof}
The proof follows from an embedding of the ring $O_{\mathscr{H}(k)}$ ($O_{\mathbf{F}_p(C)}$, resp.)
in its field of fractions $\mathscr{H}(k)$  ($\mathbf{F}_p(C)$, resp.) and Lemma \ref{lm3.2}. 
\end{proof}

\begin{lemma}\label{lm3.4}
 $\mathscr{H}^{m-1}(k)\cong\mathbf{F}_{p^{m-1}}(C).$
\end{lemma}
\begin{proof}
 (i) Recall that if the curve $C_{sing}$ has  a singular point of multiplicity $m\ge 2$,
 then its desingularization    $C_{sing}\buildrel p\over\dashrightarrow C$ consists
 of the composition of $m-1$ blow-ups. In view of Theorem \ref{thm2.1}
 and using the notation  $\mathscr{H}^{i+1}(k):=\mathscr{H}(\mathscr{H}^i(k))$, 
 one gets a tower of the Hilbert class fields:
\begin{equation}\label{eq3.8}
\mathscr{H}(k)\subset \mathscr{H}^2(k)\subset\dots \subset \mathscr{H}^{m-1}(k)\cong  \mathscr{H}^{m}(k),
\end{equation}
where the field  $\mathscr{H}^{m-1}(k)$ is the first in the tower to have the class number one. 
 
 \bigskip
 (ii)  On the other hand, (\ref{eq3.8}) and Corollary \ref{cr3.3} imply  the following tower of the function fields:
\begin{equation}\label{eq3.9}
\mathbf{F}_p(C)\subset \mathbf{F}_{p^2}(C)\subset \dots \subset \mathbf{F}_{p^{m-1}}(C)\cong  \mathbf{F}_{p^{m-1}}(C). 
\end{equation}

\bigskip
(iii)   Notice that the geometric extensions of the function fields in the tower (\ref{eq3.9}) are precluded,  since 
the intermediate curves $\{C_i ~|~1\le i\le m-2\}$ in the blow-up sequence  $C_{sing}:=C_1\buildrel p\over\dashrightarrow\dots  C_{m-2}\buildrel p\over\dashrightarrow C$
are singular.  Thus they cannot be used to define a geometric extension of the field $F_p(T)$.

\bigskip
Comparing (\ref{eq3.8}) and (\ref{eq3.9}), one gets  the isomorphism $\mathscr{H}^{m-1}(k)\cong\mathbf{F}_{p^{m-1}}(C).$
Lemma \ref{lm3.4} is proved.
 \end{proof}

\bigskip
Theorem \ref{thm1.1} follows from Lemma \ref{lm3.4}.

\bigskip\bigskip
\section{Complex multiplication revisited}

\bigskip\noindent
Let $k=\mathbf{Q}(\sqrt{-d})$ be an imaginary quadratic number field. 
Recall that the theory of complex multiplication developed by Hilbert and Weber
constructs  generators of  the field $\mathscr{H}(k)$ as follows. 

\bigskip
Denote by $\mathscr{E}_{\tau}=\mathbf{C}/(\mathbf{Z}+\mathbf{Z}\tau)$
an elliptic curve of the complex modulus $\tau=\sqrt{-d}$.  Notice that such curves correspond to the
case $g=1$ in formula (\ref{eq3.1}) and the lattice $\Lambda=\mathbf{Z}+\mathbf{Z}\tau$ has complex multiplication 
by the ring $O_k$ of algebraic integers of the field $k$.  Consider the $j$-invariant of   $\mathscr{E}_{\tau}$ given by the convergent power series:
\begin{equation}\label{eq4.1}
j(\tau)=\frac{1} {q}+744+196884q+\dots, \qquad\hbox{where}\quad q=e^{2\pi i\tau}. 
\end{equation}

\bigskip
Roughly speaking, the main theorem of complex multiplication says 
that:
\begin{equation}\label{eq4.0}
\mathscr{H}(k)=k\left(j(\sqrt{-d})\right).
\end{equation}

\bigskip
In other words, the  value of the transcendental function (\ref{eq4.1}) 
at the point $\tau=\sqrt{-d}$
is an algebraic number  generating  the Hilbert class field of   $k$.  

\bigskip
On the other hand, the  Drinfeld modules indicate  that  the generators 
of $\mathscr{H}(k)$ do not require   dramatic formulas  like   (\ref{eq4.1})
but instead  good old exponents will do.  In this section we construct such 
generators using the torsion submodules of the Drinfeld module and  Theorem \ref{thm1.1}. 
The corresponding exponential function was first obtained  in \cite{Nik0}.

\bigskip
Recall that the Drinfeld module defines explicit generators of the abelian extensions of the 
function fields (Section 2.3.3). Such modules generalize the Carlitz modules used in  construction of the 
cyclotomic function fields, and can be viewed as an analog of the exponential function $E(x)=e^{x}$.
Namely,  the torsion module $\Lambda_{\rho}[a]\in\overline{\mathbf{F}_q(C)}$ is a generator of the Hilbert class field of 
the function field $\mathbf{F}_q(C)$, where $a\in  \mathbf{F}_q(C)$ and $\rho$ is a homomorphism given by formula
(\ref{eq2.7}),  see Theorem \ref{thm2.3}.  

\bigskip
We are looking for an image of the $\Lambda_{\rho}[a]$ under the 
isomorphism $\mathbf{F}_q(C)\cong \mathscr{H}(k)$ of Theorem \ref{thm1.1} by considering  a family of the
exponential functions of the form:
\begin{equation}\label{eq4.2}
E_{\mu}(x)=\mu e^{2\pi x},
\end{equation}
where $\mu$ is a scaling parameter depending on the arithmetic of the ground field $k$.   
The following result solves the problem for the imaginary quadratic fields $k=\mathbf{Q}(\sqrt{-d})$. 
\begin{theorem}\label{thm4.1}
{\bf (\cite{Nik0})}
The image of the  torsion module $\Lambda_{\rho}[a]$  is given by the exponential function (\ref{eq4.2}) 
evaluated at the point $x=\sqrt{-d}$, where $\mu=\log \varepsilon$ and $\varepsilon$ is the fundamental unit
of the real quadratic number field  $\mathbf{Q}(\sqrt{d})$. In other words, 
\begin{equation}\label{eq4.3}
 \mathscr{H}(k)=k\left(e^{2\pi i\sqrt{d}+\log\log\varepsilon}\right)
\end{equation}
for all but a finite number of values of the square-free discriminant $d$ 
of the imaginary  quadratic field $k=\mathbf{Q}(\sqrt{-d})$. 
\end{theorem}
\begin{remark}
Formula (\ref{eq4.3})  was proved using the  Sklyanin algebras ([Stafford \& van den Bergh 2001] \cite[Example 8.3]{StaVdb1})
and the corresponding noncommutative tori (\cite[Section 1.1]{N}).  We refer the reader to \cite{Nik0} for the details. 
\end{remark}

\bigskip
The image of the  torsion module $\Lambda_{\rho}[a]$ for other number fields 
can be described as follows. 
\begin{conjecture}\label{cnj4.2}
Let $\theta$ be an algebraic number given by its minimal polynomial
$p(x)= x^n-a_1x^{n-1}+\dots -a_1x +a_0$, where the integers $a_i$ are non-negative. 
Denote by $\varepsilon$ a unit of the splitting field of the polynomial 
$q(x)=x^n-a_1x^{n-1}-\dots -a_1x-a_0$. 
If  $k\cong\mathbf{Q}(\theta)$ is a number field, then:
\begin{equation}\label{eq4.4}
 \mathscr{H}(k)=k\left(e^{2\pi\theta+\log\log\varepsilon}\right).
\end{equation}
\end{conjecture}

\bigskip
\begin{remark}
We refer the reader to \cite[pp. 187-188]{N} for the motivation of Conjecture \ref{cnj4.2}.
\end{remark}

\bigskip
\begin{example}
Let $p(x)=x^2+d$, where $d>0$ is a square-free integer. One gets $\theta=i\sqrt{d}$
and the polynomial $q(x)=x^2-d$ has the real quadratic splitting field $\mathbf{Q}(\sqrt{d})$. 
It is easy to see, that in this case (\ref{eq4.4}) coincides with the formula (\ref{eq4.3}).   
\end{example}

\bibliographystyle{amsplain}

\begin{thebibliography}{99}



\bibitem{ArtWap1}
 E.~Artin and G.~ Whaples, 
 \textit{Axiomatic characterization of fields by the product formula for valuations,} 
 Bull. Amer. Math. Soc. {\bf 51} (1945), 469-492. 


\bibitem{B}
 B.~Blackadar, \textit{$K$-Theory for Operator Algebras}, MSRI Publications,
 Springer, 1986.


\bibitem{DedWeb1}
 R.~Dedekind and H.~Weber,
 \textit{Theorie der algebraischen Functionen einer Ver\"anderlichen},  
 J. Reine Angew. Math. {\bf 92}  (1882), 181-290.


\bibitem{GMS}
G. van der Geer, B.~Moonen and R.~Schoof (Eds), 
\textit{Number Fields and Function Fields -- Two Parallel Worlds},
Progress in Mathematics {\bf 239},  Birkh\"auser,  2005. 

\bibitem{Gra1}
A. Grassi, 
\textit{Birational geometry old and new},
Bull. Amer. Math. Soc. {\bf 46} (2009), 99-123. 


\bibitem{K}
J.~Koll\'ar,  \textit{Lectures on Resolution of Singularities},
Annals of Mathematics Studies {\bf 166}, Princeton, 2007.  

\bibitem{Nik0}
I.~V.~Nikolaev \textit{On algebraic values of function $\exp~(2\pi ix +\log\log y)$},  Ramanujan J. {\bf 47} (2018), 417–425.



\bibitem{N}
I.~V.~Nikolaev, \textit{Noncommutative Geometry}, Second Edition,
De Gruyter Studies in Math. {\bf 66}, Berlin, 2022.


\bibitem{Nik1}
I.~V.~Nikolaev, \textit{Birational geometry of quaternions}, 
arXiv:2212.06555 


\bibitem{R}
M.~Rosen, \textit{Number Theory in Function Flelds},
GTM {\bf 210}, Springer,  2002. 

\bibitem{StaVdb1}
J.~T.~Stafford and M.~van ~den ~Bergh, \textit{Noncommutative curves and noncommutative
surfaces}, Bull. Amer. Math. Soc. {\bf 38} (2001), 171-216. 


\end{thebibliography}


\end{document}